\documentclass[14pt]{amsart}
\address{\newline{\normalsize Kavli IPMU (WPI), The University of Tokyo, 5-1-5 Kashiwanoha, Kashiwa, 277-8583, Japan}
\newline{\it E-mail address}: ilya.karzhemanov@ipmu.jp}
\usepackage{amscd,amsthm,amsmath,amssymb}
\usepackage[matrix,arrow]{xy}

\makeatletter\@addtoreset{equation}{section}\makeatother
\renewcommand{\theequation}{\thesection.\arabic{equation}}

\renewcommand{\thesubsection}{\bf\thesection.\arabic{equation}}
\makeatletter\@addtoreset{subsection}{equation}\makeatother

\newtheorem{theorem}[equation]{Theorem}
\newtheorem{prop}[equation]{Proposition}
\newtheorem{lemma}[equation]{Lemma}
\newtheorem{cor}[equation]{Corollary}

\theoremstyle{remark}
\newtheorem{remark}[equation]{Remark}

\theoremstyle{definition}

\newtheorem{ex}[equation]{Example}
\newtheorem*{que}{Question}

\newcommand{\cel}{\mathbb{Z}}
\newcommand{\com}{\mathbb{C}}

\newcommand{\fie}{\textbf{k}}

\newcommand{\f}{\mathbb{F}}
\newcommand{\p}{\mathbb{P}}
\newcommand{\aut}{\text{Aut}}
\newcommand{\map}{\longrightarrow}

\newcommand{\ve}{\mathcal{E}}

\newcommand{\rk}{\text{rk}}
\newcommand{\pic}{\text{Pic}}

\title{Some pathologies of Fano manifolds in positive
characteristic}

\author{Ilya Karzhemanov}

\begin{document}

\begin{abstract}
We construct examples of Fano manifolds, which are defined over a
field of positive characteristic, but not over $\com$.
\end{abstract}

\maketitle

\bigskip

\section{Introduction}
\label{section:intro}

\refstepcounter{equation}
\subsection{}
\label{subsection:intro-0}

Let $X$ be a Fano variety defined over a ground field
$\fie=\bar{\fie}$ (see \cite{isk-pro} or \cite[Ch.
V]{kollar-rat-curves} for the basic notions and facts). We will
also assume $X$ to be smooth.

The main concern of the present note is the following:

\begin{que}
Suppose $p:=\text{char}~\fie>0$. Does there exist $X$ non-liftable
to $\text{char} = 0$ (hence, in particular, $H^2(X,T_X)\ne 0$),
i.e. is it true that there is always some variety $X_0$, defined
over a $\cel$-subalgebra $R\subset\com$ of finite type, so that $X
= X_0\otimes_{R}\fie$ for $\fie := R/\frak{p}$ and a prime ideal
$\frak{p}\subset R$ with $\frak{p}\cap\cel = (p)$?
\end{que}

The problem is trivial for $\dim X = 1$, a bit trickier when $\dim
X = 2$ (but still one does not obtain any new del Pezzo surfaces
in positive characteristic due to e.g. \cite[Ch. IV,\S
2]{manin-cubic-forms}), and does not seem to get much attention in
the case of $\dim X\ge 3$.

\begin{remark}
\label{remark:van-straten} In \cite{cynk-van-straten}, many (both
non- and algebraic) Calabi-Yau $3$-folds, defined over $\f_p$,
have been constructed. The idea was to take a \emph{rigid} (i.e.
with $H^1(T) = 0$) nodal CY variety $\mathcal{X}$, defined over
$\cel$, in such a way that its mod $p$ reduction $X_p :=
\mathcal{X}\otimes_{\cel}\f_p$ acquires an additional node. Then
the needed CY $3$-fold $X$ is constructed via a small resolution
$X \map X_p$ (see \cite[Theorem 4.3]{cynk-van-straten}). We will
discuss other (general type) examples in
Remark~\ref{remark:brief-his} below.
\end{remark}

Recall that partial classification of Fano $3$-folds $X$, subject
to the ``numerical" constraints $\rk\,\text{Pic}=1$ and either
$p>5$ or $D\cdot c_2(X)>0$ for any ample divisor $D$ on $X$, was
obtained in \cite{shef-bar}.\footnote{Actually, in the light of
Corollary~\ref{theorem:main-cor-1} below, for the arguments in
\cite{shef-bar} to carry on (cf. \cite[Theorem 1.4]{shef-bar}) one
has to \emph{assume} in addition that $H^1(X,\mathcal{O}_X(-D))=0$
for any (or at least those considered in \cite{shef-bar}) ample
divisors $D$ on $X$.} As it turns out, all these $X$ do admit a
lifting to $\com$, which is insufficient for us (cf.
\cite{balico}, \cite{dejong-starr}, \cite{megyesi}, \cite{saito}).

Thus, in order to find $X$ with an interesting behavior in
positive characteristic one should turn to the ``qualitative"
properties of Fano varieties, as the following example suggests
(compare with Remark~\ref{remark:van-straten}):

\begin{ex}
\label{example:easy-ex} The hypersurface $Y_p := (\sum x_i^py_i =
0) \subset \p^n\times\p^n$ (see
\cite[(1.4.3.2)]{kollar-rat-curves}) is a Fano variety iff $p\le
n$. Furthermore, $Y_p$ is an (obvious) $SL(n+1)$-homogeneous
space, which can deform to a non-homogeneous variety. This
violates a similar property of homogeneous spaces defined over
$\fie\subseteq\com$. On the other hand, one may consider quotient
schemes $G/P$, where $G$ is a semi-simple algebraic group and
$P\subset G$ is its \emph{non-reduced} parabolic subgroup. (These
$G/P$ are rational for $p>3$ due to \cite{wenzel}.) Yet,
unfortunately, both of the examples lift to characteristic $0$, at
least when $\rk~\text{Pic}=1$ (cf. the discussion in
\cite{shef-bar}).\footnote{Another possible approach to answer
{\bf Question} affirmatively is via the \emph{wild} conic bundle
structures (for $p=2$) on Fano manifolds (see \cite[Theorem
5.5]{shef-bar}). But again one does not get anything new in
characteristic $2$ due to \cite[Theorem 7]{mori-saito} (see also
\cite{totaro-1} for some discussion and results on (birational)
geometry of wild conic bundles).} Finally, let us mention
$p$-coverings $X$ of Fano hypersurfaces in $\p^{n+1}$, for which
$\wedge^{n-1}\Omega_X^1$ contains a positive line subbundle (see
\cite{kollar-non-rat-hypers}). But again, even though this
property of $\wedge^{n-1}\Omega_X^1$ is specific to positive
characteristic (see e.g. \cite{fedya}), $X$ is (obviously)
liftable to $\com$.
\end{ex}

We answer our {\bf Question} via the next

\begin{theorem}
\label{theorem:main} For any $p\ge 3$ and $d,N\gg 1$, there exists
a flat family $\mathcal{X}\map B_d$ of Fano $3$-folds,
parameterized by an algebraic variety $B_d$ of dimension $2d$,
such that

\begin{itemize}

    \item the fiber $X := \mathcal{X}_b$ is non-liftable to $\text{char} = 0$ for
    generic $b\in B_d$;

    \smallskip

    \item the fibers $\mathcal{X}_b$ and $\mathcal{X}_{b'}$ are non-isomorphic for generic
    $b\in B_d,
    b'\in B_{d'}$, with $d\ne d'$;

    \smallskip

    \item $(-K_{\mathcal{X}_b}^3) \ge N$ for the anticanonical degree of $X = \mathcal{X}_b$.

\end{itemize}

\end{theorem}

\begin{remark}
\label{remark:brief-his} Examples of manifolds (in arbitrary
dimension) of general type violating Kodaira vanishing -- more
precisely, having $H^1(L^{-1})\ne 0$ for an ample line bundle $L$,
-- have been constructed in \cite{mukai}. This was a development
of the method from \cite{raynaud}, where a surface $S$ of general
type (with $H^1(S,L^{-1})\ne 0$) is constructed together with a
morphism $S \map C$ onto a curve, so that the
\emph{$\fie(C)$-curve} $S$ violates the Mordell-Weill property. We
follow these trends starting from {\ref{subsection:const-1}}
below, for the $3$-fold $X$ we need is constructed by finding a
$\p^1$-bundle $W$ over $S$ first, with $-K_W$ nef, and then taking
a $p$-covering $X\map W$ \`a la \cite{ekedahl}.
\end{remark}

Next corollary complements the results mentioned in
Remarks~\ref{remark:van-straten}, \ref{remark:brief-his}:

\begin{cor}
\label{theorem:main-cor} For generic $X$ as in
Theorem~\ref{theorem:main} and $m\gg 1$, there is a cyclic
$m$-covering $\pi_m: X_m\map X$ such that $X_m$ is smooth,
$K_{X_m}$ is nef, and $X_m$ is non-liftable to $\text{char} = 0$.
\end{cor}

The proof of Theorem~\ref{theorem:main} also yields (compare with
\cite{takayama})

\begin{cor}
\label{theorem:main-cor-1} In the notations from
Corollary~\ref{theorem:main-cor}, irregularity $q(X) \ne 0$, as
well as $q(X_m) \ne 0$.
\end{cor}

(Both of Corollaries~\ref{theorem:main-cor},
\ref{theorem:main-cor-1} are proved at the end of
Section~\ref{section:exa}.)

\refstepcounter{equation}
\subsection{}
\label{subsection:intro-1}

Theorem~\ref{theorem:main} and Corollaries~\ref{theorem:main-cor},
\ref{theorem:main-cor-1} seem to be the ``minimal" illustrations
of pathological behavior of Fano varieties in positive
characteristic (cf. the discussion after
Remark~\ref{remark:van-straten}). Note also that
Theorem~\ref{theorem:main} provides an unbounded family of
(smooth) Fano $3$-folds (cf. \cite{ko-mi-mo}, \cite{bir-unbound}).
Furthermore, it follows easily from the arguments after
Lemma~\ref{theorem:van-th} below that the cones
$\overline{NE}(X_b)$, with varying $b \in B_d$, admit infinitely
many ``jumps" (compare with \cite{totaro-jump}, \cite{hacon},
\cite{wies}).

Thus, one can see that essentially every ``standard" property of
Fano manifolds breaks in positive characteristic, the fact which
is due (in our opinion) to the following principle behind (compare
with Remark~\ref{remark:char-p-pathology} below):
$$
\text{\it every object over}\ \fie, \text{\it char}~\fie>0,\
\text{\it is defined \emph{only up to} the Frobenius
twist}.\footnote{As, for example, the hyperelliptic curve $y^2=x^p
- a,a\in\fie$, of genus $(p-1)/2$, covered (after Frobenius twist)
by a rational curve (see \cite[Corollary 1]{tate}).}
$$
(This was probably first exploited in \cite{tate}, while the
ultimate reading on the subject are \cite{mumford-pat-1},
\cite{mumford-pat-2} and \cite{mumford-pat-3} of course.)

Finally, it would be interesting to work out the constructions in
{\ref{subsection:const-2}}, taking a (supersingular) Kummer
surface in place of $S$ (cf. \cite{bog-tsch}).

\bigskip

\section{The construction}
\label{section:exa}

\refstepcounter{equation}
\subsection{}
\label{subsection:const-0}

Fix a smooth projective surface $S$. We will denote by
$\mathcal{E}$ (resp. $\mathcal{F}$) a vector bundle (resp. a
coherent sheaf) on $S$. Let us recall some standard notions and
facts about $\mathcal{E}$ and $\mathcal{F}$ (see e.g.
\cite{don-kron}, \cite{fried}, \cite{okonek-et-al}).

First of all, one defines the Chern classes $c_i :=
c_i(\mathcal{E})\in A^i(S)$ for $\mathcal{E}$ (and similarly for
$\mathcal{F}$, using a locally free resolution), with $c_1 =
\det\ve$. In fact, letting $W := \p(\ve)$ there is a natural
inclusion $A^*(S)\hookrightarrow A^*(W)$ of groups of cycles
induced by the projection $\pi: W \map S$, and the following
identity (called the \emph{Hirsch formula}) holds:
$$
H^2 + H\cdot c_1 + c_2 = 0.
$$
Here $H := \mathcal{O}(1)\in A^1(W)$ is the Serre line bundle on
$W$ and $c_i$ are identified with $\pi^*(c_i)$ (cf.
Remark~\ref{remark:p-m-rem}). In particular, for $r := \rk~\ve =
2$ we have
$$
H^3 = -c_1^2 - c_2,
$$
which together with the Euler's formula
\begin{equation}
\nonumber K_W = -rH + \pi^*(K_S + c_1)
\end{equation}
gives
$$
(-K_W^3) = 6K_S^2 + 10c_1^2 + 24K_S\cdot c_1 - 8c_2.
$$

\begin{remark}
\label{remark:p-m-rem} When $\fie\subseteq\com$ (so that $c_i\in
H^{2i}(S,\cel)$) and the structure group of $\ve$ is not $SU(2)$
(or $H^1(S,\cel)\ne 0$), the previous formulae (except for the
Euler's one and with suitably corrected $(-K_W^3)$) should be read
with all the $c_i$ up to $\pm$. In fact, when $\ve\simeq \ve^*$
(the dual of $\ve$) and $\simeq$ is non-canonical, there is no
preference in choosing $\pi^*(c_i)$ or $-\pi^*(c_i)$ as one may
change the orientation in the fibers of $\ve$. In particular, the
Hirsch formula turns into $H^2 \pm H\cdot c_1 \pm c_2 = 0$, with
both $c_i$ having a definite sign at once.
\end{remark}

Further, let $Z\subset S$ be a $0$-dimensional subscheme supported
at a finite number of points $p_1,\ldots,p_m$. Put $\ell_i :=
\dim~\mathcal{O}_{S,p_i}/I_{Z,p_i}$ and $\ell(Z):=\sum_i\ell_i$
(the \emph{length} of $Z$) for the defining ideal $I_Z$ of $Z$.
One can show that $c_2(j_*\mathcal{O}_Z) = -\ell(Z)\in A^2(S)$,
where $j:Z\map X$ is the inclusion map. In particular, if $\ve$
admits a splitting
\begin{equation}
\label{split} 0 \to \mathcal{L} \to \ve \to \mathcal{L}'\otimes
I_Z\to 0
\end{equation}
for some line bundles $\mathcal{L},\mathcal{L}'\in\pic(S)$, we
obtain (using the Whitney's formula)
\begin{equation}
\label{cherns} c_1 = c_1(\mathcal{L})+c_1(\mathcal{L}'),\qquad c_2
= c_1(\mathcal{L})\cdot c_1(\mathcal{L}') + \ell(Z).
\end{equation}

\begin{ex}
\label{example:sec-split} Let $s\in H^0(S,\ve)$ be a section.
Assume for simplicity that the zero locus $(s)_0\subset S$ of $s$
has codimension $\ge 2$. Then one gets an exact sequence of
sheaves $0\to \mathcal{O}_S\stackrel{s}{\map}\ve\stackrel{\wedge
s}{\map}\bigwedge^2\ve\otimes I_{(s)_0}\to 0$ (cf. \eqref{split}).
\end{ex}

Suppose now that $r=2$. Let $\mathcal{L}\subset\ve$ be a line
subbundle such that the sheaf $\ve/\mathcal{L}$ is torsion-free.
Then by working locally it is easy to obtain an exact sequence
\eqref{split}. All such sequences (\emph{extensions} of $\ve$) are
classified by the group $\text{Ext}^1(\mathcal{L}'\otimes
I_Z,\mathcal{L})$, which in the case when $Z$ is a locally
complete intersection coincides with $\mathcal{O}_Z$. (More
precisely, since $\mathcal{L},\mathcal{L}'$, etc. are defined up
to the $\fie^*$-action, the classifying space for the extensions
is the projectivization $\p(\text{Ext}^1(\mathcal{L}'\otimes
I_Z,\mathcal{L})) = \p(\mathcal{O}_Z)$.) Note also that
\eqref{split} provides a \emph{locally free} extension when
$H^2(S,\mathcal{L}'^{-1}\otimes \mathcal{L}) = 0$ (Serre's
criterion).

\refstepcounter{equation}
\subsection{}
\label{subsection:const-1}

Let $Y$ be a smooth projective variety with a line bundle $L\in
\text{Pic}(Y)$. Frobenius $F: Y \map Y$ induces a homomorphism
$$F^*: H^1(Y,L^{-1}) \to H^1(Y,L^{-p})$$ and one has
$$
\text{Ker}\,F^* \simeq \{df\in\Omega_{\fie(Y)}\ \vert\
f\in\fie(Y),(df)\ge pD \}
$$
once $L = \mathcal{O}_Y(D)$ for $D\ge 0$ (see \cite[Theorem
1]{mukai}). We may assume that $\text{Ker}\,F^*\ne 0$ (see
\cite[Theorem 2]{mukai}), $H^1(Y,\mathcal{O}_Y(-pD))=0$ for an
ample $D$ (cf. the proof of \cite[Lemma 1.4]{shef-bar}), which
yields a non-split extension
$$
0\to \mathcal{O}_Y(-D) \to \ve_{L,Y}\to \mathcal{O}_Y \to 0
$$
such that the corresponding extension for $F^*\ve_{L,Y}$ splits.
This defines a $\mathbb{G}_a$-torsor over $Y$ (w.r.t. the line
bundle $L^{-1}$) and a subscheme $X \subset \p(\ve_{L,Y})$ which
projects ``$p$-to-$1$" onto $Y$. More precisely, the morphism
$\phi: X \map Y$ is purely inseparable  of degree $p$, so that $X$
is regular and $K_X = \phi^*(K_Y - (p-1)D)$ (see \cite[Proposition
1.7,\,(16)]{mukai}).

\begin{ex}
\label{example:ray-stuff} Let $P(y)$ be a polynomial of degree $e$
and $C\subset\mathbb{A}^2$ a plane curve given by the equation
$P(y^p) - y = z^{pe-1}$ (see \cite[Example 1.3]{mukai}). One
easily checks that $(dz)=pe(pe-3)(\infty)$ for $C\subset\p^2$
identified with its projectivization and $\infty\in C$. Then,
taking $y^pz$ (with $d(y^pz) = y^pdz$) and $e\gg 1$, we obtain
$\dim\text{Ker}\,F^* = h^1(C,(3-pe)(\infty))\ge 2$ for $D :=
(pe-3)(\infty)$. Moreover, \cite[Propositions 2.3, 2.6]{mukai}
shows that instead of $C$ we may take $Y$ as above, with
$h^1(Y,L^{-1})\ge 2$, ample $K_Y$ (for $p\ge 3$) and arbitrary
$\dim Y\ge 2$.
\end{ex}

\refstepcounter{equation}
\subsection{}
\label{subsection:const-2}

Now let $S:=Y$ be the \emph{Raynaud surface}. Recall that $S$ can
be realized as a double cover of the $\p^1$-bundle $\p(\ve_{L,C})$
over the curve $C$ from Example~\ref{example:ray-stuff}, ramified
along a section and the $p$-cover $X$ of $C$ as above (run the
constructions in \cite[2.1]{mukai} with $k:=2$). Let $\psi: S\map
C$ be the induced rational curve fibration.

\begin{prop}
\label{theorem:exp-for-c-i-k-nef} In the previous setting, there
exists a vector bundle $\ve$ on $S$ such that

\begin{itemize}

    \item $r=2,c_2=0$ and $c_1=-nK_S$ for an arbitrary $n\gg 1$;

    \smallskip

    \item the divisor $-K_W$ is nef on $W = \p(\ve)$ (cf. {\ref{subsection:const-0}}).

\end{itemize}

\end{prop}

\begin{proof}
In the exact sequence \eqref{split}, take $\mathcal{L}$ an
arbitrary ample, $\mathcal{L}' := -\mathcal{L}-nK_S$ and $Z$ any
finite with $\ell(Z)=-c_1(\mathcal{L})\cdot c_1(\mathcal{L}')$.
Then from \eqref{cherns} we get $c_1=-nK_S,c_2=0$, and the Euler's
formula gives
$$
-K_W = 2H + \pi^*(n-1)K_S.\footnote{Note also that
$H^2(S,\mathcal{L}'^{-1}\otimes \mathcal{L}) = H^2(nK_S) = 0$ and
so $\ve\in\text{Ext}^1(\mathcal{L}'\otimes I_Z,\mathcal{L})$ is
locally free.}
$$
We will deduce that this $-K_W$ is nef from the two forthcoming
lemmas.

Consider $l\subset S$, a general fiber of $\psi$, and put
$\Sigma_l := \pi^{-1}l = \p(\ve\big\vert_l)$.

\begin{lemma}
\label{theorem:van-th-as} We have
$(\Omega_{\Sigma_l}^2)^{**}\simeq
\omega_{\Sigma_l/l}\otimes\mathcal{L}^{-1}\otimes\mathcal{L}'$ for
the double dual of $\Omega_{\Sigma_l}^2$.
\end{lemma}

\begin{proof}
Note that $l$ is a rational curve with unique singular point $o :=
\text{Sing}(l)$ such that $l$ is given by the equation $y^2 = z^p$
(affine) locally near $o$. This shows that the sheaf
$\Omega^1_{l}$ has torsion, with support in $o$, for $2ydy =
pz^{p-1}dz = 0$. In particular, we get $(\Omega^1_{l})^{**} =
\mathcal{O}_l\cdot dz$ near $o$, with $dy$ being the second
(torsion) generator of $\Omega^1_{l}$. This easily yields
$$
h^0(l,(\Omega^1_{l})^{**}) \ge \frac{(p-1)(p-2)}{2} -
\frac{(p-2)(p-3)}{2} - 1\ge 0.
$$
Indeed, if $\tilde{l}\subset\p^2$ is the closure of the curve
$(y^2 = z^p)$, then $\text{Sing}(\tilde{l}) = \{o,\infty\}$ and
there is a natural morphism $l \map \tilde{l}$. The above estimate
$h^0(l,(\Omega^1_{l})^{**}) \ge 0$ now follows because $l \map
\tilde{l}$ is the normalization near $\infty$ and
$\text{mult}_{\infty}(\tilde{l}) = p - 2$.

Further, we have $$\p(\ve\big\vert_l) =
\p(\mathcal{L}\big\vert_l\oplus\mathcal{L}'\big\vert_l)$$ with a
local fiber coordinate $t$, so that (K\"ahler) $dt$ induces a
section $\tau$ of $(\Omega_{\Sigma_l}^2)^{**}$. Namely, given the
indicated properties of $\ve\big\vert_l$ and $\Omega^1_{l}$, we
obtain $$dt\wedge\xi=\tau\big\vert_{U\cap U'} =
fdt\wedge\xi\big\vert_{U'\cap U}$$ for some $\xi\in
H^0(l,(\Omega^1_{l})^{**})$ and $f\in \mathcal{O}(U\cap U'\cap l)$
generating a cycle in $H^1(l,\mathcal{O}^*_l)$. (Here
$U,U'\subset\Sigma_l$ are arbitrary open charts pulled back from
$l$, and the latter is identified with the zero-section of
restrictions $\ve\big\vert_U,\ve\big\vert_{U'}$.)

More precisely, we get that $\tau$ lifts to a section of
$(\mathcal{L}^{-1}\otimes\mathcal{L}')\big\vert_l$, which gives
$(\Omega_{\Sigma_l}^2)^{**}\simeq
\omega_{\Sigma_l/l}\otimes\mathcal{L}^{-1}\otimes\mathcal{L}'$.
\end{proof}

\begin{remark}
\label{remark:char-p-pathology} The proof of
Lemma~\ref{theorem:van-th-as} illustrates the following phenomenon
typical to $\text{char}>0$: the sheaf $\Omega_X^1$ may be rank $1$
torsion, thus not equal to $\text{Hom}(T_X,\mathcal{O}_X)$, where
$\dim T_{X,x}\ge 2$ for some $x\in \text{Sing}(X)$. Let us also
point out that the interrelation between $\Omega_{\Sigma_l}^2$ and
canonical bundle of the normalization of $\Sigma_l$ is not clear
(meaning the conductor cycle need not be effective) because
$\Sigma_l$ violates the $S_2$-property (for the curve $l$
obviously does, -- the rational function $y^{1/p}: l\map\p^1$ is
defined at every point near $o$, but is not \emph{regular} at
$o$).
\end{remark}

\begin{lemma}
\label{theorem:van-th-a} The linear system
$|-K_W\big\vert_{\Sigma_l}|$ is basepoint-free.
\end{lemma}

\begin{proof}
It follows from Lemma~\ref{theorem:van-th-as} that the line bundle
$(\Omega_{\Sigma_l}^2)^*$ satisfies
$$H^0(\Sigma_l,(\Omega_{\Sigma_l}^2)^*)\supseteq H^0(\Sigma_l,\mathcal{L}\otimes\mathcal{L}'^{-1}).$$
Indeed, since $\p^1$ normalizes $l$, we get $\Omega^1_{\p^1} =
\mathcal{O}_{\p^1}\cdot dz$ near (the preimage of) $o$ (cf. the
proof of Lemma~\ref{theorem:van-th-as}). Hence the sheaf
$\omega_{\Sigma_l\slash l}$ pulls back to
$\omega_{\Sigma\slash\p^1}$ for the ruled surface $\Sigma :=
\p(\ve\big\vert_{\p^1})$. It remains to notice that (the class of)
$(\omega_{\Sigma\slash\p^1})^* = 2[\text{minimal section on}\
\Sigma] + 2[\text{fiber}]$ is ample.

On the other hand, for $s\in H^0(\Sigma_l,
(\Omega_{\Sigma_l}^2)^*)$ we get
$$
-K_W\big\vert_{\Sigma_l} = (s=0) + \Sigma_l\big\vert_{\Sigma_l} =
(s=0)
$$
by adjunction, and the claim follows.
\end{proof}

\begin{lemma}
\label{theorem:van-th} $H^1(W,-K_W-\Sigma_l) = 0$.
\end{lemma}

\begin{proof}
Notice that $$H^1\big(W,\pi_*(2H + \pi^*(n-1)K_S - \Sigma_l)\big)
= H^1\big(S,S^2\ve\otimes\mathcal{O}_S((n-1)K_S - l)\big) = 0$$ by
Serre vanishing (for $K_S$ is ample). Furthermore, we have
$$R^1\pi_*(2H + \pi^*(n-1)K_S-\Sigma_l)=0$$ by Grothendieck
comparison, since the divisor $2H + \pi^*(n-1)K_S - \Sigma_l$ is
$\pi$-ample. The assertion now follows from the Leray spectral
sequence.
\end{proof}

From the defining exact sequence for $\Sigma_l$ and
Lemma~\ref{theorem:van-th} we get a surjection
$$
H^0(W,-K_W)\twoheadrightarrow
H^0(\Sigma_l,-K_W\big\vert_{\Sigma_l}).
$$
Then by Lemma~\ref{theorem:van-th-a} the base locus of $|-K_W|$ is
contained in the fibers of $\pi$. This shows that $-K_W$ is nef
and finishes the proof of
Proposition~\ref{theorem:exp-for-c-i-k-nef}.
\end{proof}

Recall that $H^1(S,L^{-1})\ne 0$ for $L = \mathcal{O}_S(D)$ (cf.
the beginning of {\ref{subsection:const-2}}). One can also observe
that $H^1(W,\pi^*L^{-1})\simeq H^1(S,L^{-1})$. Hence, setting $Y
:= W$ (resp. $L := \pi^*L$) in the notations of
{\ref{subsection:const-1}}, we can pass to the corresponding
$p$-cover $\phi: X = \p(\ve_{\pi^*L,W}) \map W$. The divisor
$-K_X$ is ample by Proposition~\ref{theorem:exp-for-c-i-k-nef} and
Kleiman's criterion (see \cite[Ch. I, Theorem
8.1]{hart-ample-subvars}).

Finally, since $W = \p(\ve)$ for $\ve$ constructed as an extension
(cf. {\ref{subsection:const-0}}), to complete the proof of
Theorem~\ref{theorem:main} it suffices to vary $\mathcal{L}$ and
$Z$ as above in such a way that both $c_1(\ve)=-nK_S,c_2(\ve)=0$
remain fixed, while $\ell(Z)\to\infty$ (cf. \eqref{cherns} and the
properties of $\ve$ in
Proposition~\ref{theorem:exp-for-c-i-k-nef}). The assertion then
follows from the formula for $(-K_W^3)$ in
{\ref{subsection:const-0}} and the fact that the cone
$\overline{NE}(X)$ is finite polyhedral (see e.g. \cite{kollar},
\cite{keel}). More precisely, we have obtained that $X =
\mathcal{X}_b$ surjects onto $W = \p(\ve)$, with $\ve\in
\text{Ext}^1(\mathcal{L}'\otimes I_Z,\mathcal{L})$ varying
together with $Z$ (or $b\in B_d$). Then it follows from the
discussion after Example~\ref{example:sec-split} that for
different $\ve$ the $3$-folds $W$ are non-isomorphic. This gives
$\mathcal{X}_b\not\simeq \mathcal{X}_{b'}$ for generic $b'\in
B_{d'},d'\ne d$, and proves the second claim of
Theorem~\ref{theorem:main}, whereas the other two are evident from
the construction of $X$.

\begin{remark}
\label{remark:conslus} Alternatively, one could show that already
$W$ is Fano, by applying the results in \cite{bir}, \cite{xu} (and
the formula for $(-K_W^3)$ of course).
\end{remark}

We now pass on to the proof of Corollaries:

\begin{proof}[Proof of Corollary~\ref{theorem:main-cor}]
$X_m$ is obtained by the cyclic covering of $X$ with ramification
at a smooth surface from $|-mK_X|$ (see e.g. \cite{kol-mor}). This
gives $\mu_m\in\aut(X_m)$ (with $X = X_m\slash\mu_m$) and $(m,p) =
1$. We may assume the $\mu_m$-action on $X_m$ comes from a cyclic
projective action on some $\p^N\supset X_m$.

Now, if $X_m$ were liftable to $\com$, the $\mu_m$-action on
$X_m\subset\p^N$ must also lift. Then we would get that $X =
X_m\slash\mu_m$ is liftable to $\com$, a contradiction.
\end{proof}

\begin{proof}[Proof of Corollary~\ref{theorem:main-cor-1}]
Suppose that $q(X) = 0$. We are going to find an ample line bundle
$\frak{L}$ on $X$ such that $H^1(X_m,\pi_m^*\frak{L}^{-1}) \ne 0$
and $(p-1)\pi_m^*\frak{L} - K_{X_m}$ is ample. This will
contradict \cite[Ch. II, Corollary 6.3]{kollar-rat-curves}

The description of $\text{Ker}\,F^*$ in {\ref{subsection:const-1}}
and the fact that $h^1(W,\pi^*L^{-1}) = h^1(S,L^{-1})\ge 2$ (cf.
Example~\ref{example:ray-stuff}) yield
$h^1(X,(\phi\circ\pi)^*L^{-1})\ge 2$ -- for not every element in
$H^1(W,\pi^*L^{-1})$ is a $p$-power (when lifted to $X$).
Furthermore, the same argument shows that
$H^1(X_m,\pi_m^*\frak{L}^{-1})\ne 0$, provided we have found
$\frak{L}$ on $X$ as needed.

Now set $\frak{L}:=\mathcal{O}_X(-K_X + (\pi\circ\phi)^*D)$. This
is obviously ample. Moreover, by the vanishing in
\cite[Proposition 6.1]{shef-bar} and results from \cite[\S\S
3,4]{shef-bar} we may assume generic surface $S_X\in|-K_X|$ to be
smooth (hence \emph{\text{K3}-like}), which leads to an exact
sequence
$$
0 \to H^1(X,-S_X - (\pi\circ\phi)^*D)\to
H^1(X,-(\pi\circ\phi)^*D)\to H^1(S,-(\pi\circ\phi)^*D\big\vert_S).
$$
Here $H^1(S,-(\pi\circ\phi)^*D\big\vert_S)=0$ by \cite[Proposition
3.4]{mukai}, $H^1(X,-(\pi\circ\phi)^*D)\ne 0$ by the previous
considerations, $\mathcal{O}_X(K_X - (\pi\circ\phi)^*D) =
\frak{L}^{-1}$, and thus we are done.\footnote{Note that
$(p-1)\pi_m^*\frak{L} - K_{X_m}\equiv \pi_m^*\big(-(p - 1 +
\displaystyle\frac{m+1}{m})K_X + (\pi\circ\phi)^*D\big)$ is
ample.}

Hence we get $q(X) \ne 0$. Then also $q(X_m) \ne 0$ for the flat
morphism $\pi_m$.
\end{proof}

\bigskip

\thanks{{\bf Acknowledgments.}
I am grateful to C. Birkar, A.\,I. Bondal, S. Galkin, M.
Kontsevich, and Yu.\,G. Prokhorov for their interest and helpful
conversations. Also the referee's suggestions helped me to improve
the exposition. Some parts of the paper were prepared during my
visits to AG Laboratory (HSE, Moscow) and University of Miami
(US). I am grateful to both Institutions for hospitality. The work
was supported by World Premier International Research Initiative
(WPI), MEXT, Japan, and Grant-in-Aid for Scientific Research
(26887009) from Japan Mathematical Society (Kakenhi).

\bigskip

\renewcommand{\thesubsection}{\bf A.\arabic{equation}}

\renewcommand{\theequation}{A.\arabic{equation}}

\renewcommand{\thesection}{}

\section{}

\section*{Appendix}

\refstepcounter{equation}
\subsection{}
\label{subsection:com-1}

Let me elaborate further on the footnote 1) on page 1 of the
present paper. The second sentence in the proof of \cite[Theorem
1.4]{shef-bar} claims that ``By Serre vanishing, we may assume
that $H^1(X,\mathcal{O}(-pD)) = 0$ \ldots". This phrase is
actually the cornerstone for the whole \cite[Section 1]{shef-bar}
and is totally misleading (in fact false) -- it was never
explained in \cite{shef-bar} why one can apply Serre vanishing to
the \emph{fixed} $p$ and an \emph{arbitrary} ample $D$ (whereas
vanishing requires (a priori arbitrary) \emph{sufficiently large}
$p$). There are more confusing assertions in \cite[Section
1]{shef-bar} (such as e.g. the name \emph{irregularity} (and the
notation $q(X)$) for the number $h^0(X,\Omega_X^1)$). Amazingly,
the author of \cite{shef-bar} has presented same results at a
recent conference in Moscow (2017), with apparently a fare
success.

\refstepcounter{equation}
\subsection{}
\label{subsection:com-2}

Initially the present paper was accepted for publishing in
Manuscripta Mathematica. However, some time later I have received
a message from the Editors, informing me that there were several
anonymous emails which claimed that the constructions of my paper
are incorrect. Basically, there were two mathematically fruitful
extra-reports, expressing the concerns, whose content and my
responses to them are available at request. Let me stress though
that neither of the reports contained a satisfactory support for
their authors' strong acquisitions. Nevertheless, despite this
fact the Editors of Manuscripta initiated a retraction process for
the paper, and the manuscript had actually been retracted
recently. My strong disagreement and the evidence in its favor
were completely ignored. Nor I was allowed to support my paper
with the corresponding Addendum.

\refstepcounter{equation}
\subsection{}
\label{subsection:com-3}

The present situation around my paper is a result of personal
conflict and is a part of continuous manhunt on me (initiated
unfortunately by my former advisor Cheltsov). It is quite sad that
such mechanisms as scientific publishing are used, on the one
hand, in order to push forward wrong papers (cf. \cite{shef-bar}),
and suppress ones mathematical personality on the other. There is
no trace here of \emph{objective} delivering the scientific
results to the community.


\bigskip

\begin{thebibliography}{37}

\bibitem{balico}
E. Ballico, On Fano threefolds in characteristic $p$, Rend. Sem.
Mat. Univ. Politec. Torino {\bf 47} (1989), no.~1, 57--70 (1991).

\smallskip

\bibitem{bir}
C. Birkar, Existence of flips and minimal models for 3-folds in
char p, Preprint arXiv:1311.3098 (to appear in Annales
scientifiques de l'ENS).

\smallskip

\bibitem{fedya}
F. A. Bogomolov, Holomorphic tensors and vector bundles on
projective manifolds, Izv. Akad. Nauk SSSR Ser. Mat. {\bf 42}
(1978), no.~6, 1227--1287, 1439.

\smallskip

\bibitem{bog-tsch}
F. Bogomolov\ and\ Y. Tschinkel, Rational curves and points on
$K3$ surfaces, Amer. J. Math. {\bf 127} (2005), no.~4, 825--835.

\smallskip

\bibitem{cynk-van-straten}
S. Cynk\ and\ D. van Straten, Small resolutions and non-liftable
Calabi-Yau threefolds, Manuscripta Math. {\bf 130} (2009), no.~2,
233--249.

\smallskip

\bibitem{hacon}
T. de Fernex\ and\ C. D. Hacon, Rigidity properties of Fano
varieties, in {\it Current developments in algebraic geometry},
113--127, Math. Sci. Res. Inst. Publ., 59, Cambridge Univ. Press,
Cambridge.

\smallskip

\bibitem{dejong-starr}
A. J. de Jong\ and\ M. Starr, A note on Fano manifolds whose
second Chern character is positive, Preprint arXiv:math/0602644.

\smallskip

\bibitem{don-kron}
S. K. Donaldson\ and\ P. B. Kronheimer, {\it The geometry of
four-manifolds}, Oxford Mathematical Monographs, Oxford Univ.
Press, New York, 1990.

\smallskip

\bibitem{ekedahl}
T. Ekedahl, Canonical models of surfaces of general type in
positive characteristic, Inst. Hautes \'Etudes Sci. Publ. Math.
No. 67 (1988), 97--144.

\smallskip

\bibitem{fried}
R. Friedman, {\it Algebraic surfaces and holomorphic vector
bundles}, Universitext, Springer, New York, 1998.

\smallskip

\bibitem{hart-ample-subvars}
R. Hartshorne, {\it Ample subvarieties of algebraic varieties},
Lecture Notes in Mathematics, Vol. 156, Springer, Berlin, 1970.

\smallskip

\bibitem{isk-pro}
V. A. Iskovskikh\ and\ Yu.\ G. Prokhorov, Fano varieties, in {\it
Algebraic geometry, V}, 1--247, Encyclopaedia Math. Sci., 47,
Springer, Berlin.

\smallskip

\bibitem{keel}
S. Keel, Basepoint freeness for nef and big line bundles in
positive characteristic, Ann. of Math. (2) {\bf 149} (1999),
no.~1, 253--286.

\smallskip

\bibitem{kollar}
J. Koll\'ar, Extremal rays on smooth threefolds, Ann. Sci. \'Ecole
Norm. Sup. (4) {\bf 24} (1991), no.~3, 339--361.

\smallskip

\bibitem{kollar-non-rat-hypers}
J. Koll\'ar, Nonrational hypersurfaces, J. Amer. Math. Soc. {\bf
8} (1995), no.~1, 241--249.

\smallskip

\bibitem{kollar-rat-curves}
J. Koll\'ar, {\it Rational curves on algebraic varieties},
Ergebnisse der Mathematik und ihrer Grenzgebiete. 3. Folge. A
Series of Modern Surveys in Mathematics, 32, Springer, Berlin,
1996.

\smallskip

\bibitem{ko-mi-mo}
J. Koll\'ar, Y. Miyaoka\ and\ S. Mori, Rational connectedness and
boundedness of Fano manifolds, J. Differential Geom. {\bf 36}
(1992), no.~3, 765--779.

\smallskip

\bibitem{kol-mor}
J. Koll\'ar\ and\ S. Mori, {\it Birational geometry of algebraic
varieties}, translated from the 1998 Japanese original, Cambridge
Tracts in Mathematics, 134, Cambridge Univ. Press, Cambridge,
1998.

\smallskip

\bibitem{manin-cubic-forms}
Yu. I. Manin, {\it Cubic forms: algebra, geometry, arithmetic},
translated from the Russian by M. Hazewinkel, North-Holland,
Amsterdam, 1974.

\smallskip

\bibitem{megyesi}
G. Megyesi, Fano threefolds in positive characteristic, J.
Algebraic Geom. {\bf 7} (1998), no.~2, 207--218.

\smallskip

\bibitem{mori-saito}
S. Mori\ and\ N. Saito, Fano threefolds with wild conic bundle
structures, Proc. Japan Acad. Ser. A Math. Sci. {\bf 79} (2003),
no.~6, 111--114.

\smallskip

\bibitem{mukai}
S. Mukai, Counterexamples to Kodaira's vanishing and Yau's
inequality in positive characteristics, Kyoto J. Math. {\bf 53}
(2013), no.~2, 515--532.

\smallskip

\bibitem{mumford-pat-1}
D. Mumford, Pathologies of modular algebraic surfaces, Amer. J.
Math. {\bf 83} (1961), 339--342.

\smallskip

\bibitem{mumford-pat-2}
D. Mumford, Further pathologies in algebraic geometry, Amer. J.
Math. {\bf 84} (1962), 642--648.

\smallskip

\bibitem{mumford-pat-3}
D. Mumford, Pathologies. III, Amer. J. Math. {\bf 89} (1967),
94--104.

\smallskip

\bibitem{bir-unbound}
T. Okada, On the birational unboundedness of higher dimensional
$\Bbb Q$-Fano varieties, Math. Ann. {\bf 345} (2009), no.~1,
195--212.

\smallskip

\bibitem{okonek-et-al}
C. Okonek, M. Schneider\ and\ H. Spindler, {\it Vector bundles on
complex projective spaces}, corrected reprint of the 1988 edition,
Modern Birkh\"auser Classics, Birkh\"auser/Springer Basel AG,
Basel, 2011.

\smallskip

\bibitem{raynaud}
M. Raynaud, Contre-exemple au ``vanishing theorem'' en
caract\'eristique $p>0$, in {\it C. P. Ramanujam---a tribute},
273--278, Tata Inst. Fund. Res. Studies in Math., 8, Springer,
Berlin.

\smallskip

\bibitem{saito}
N. Saito, Fano threefolds with Picard number 2 in positive
characteristic, Kodai Math. J. {\bf 26} (2003), no.~2, 147--166.

\smallskip

\bibitem{shef-bar}
N. I. Shepherd-Barron, Fano threefolds in positive characteristic,
Compositio Math. {\bf 105} (1997), no.~3, 237--265.

\smallskip

\bibitem{takayama}
Y. Takayama, Raynaud-Mukai construction and Calabi-Yau threefolds
in positive characteristic, Proc. Amer. Math. Soc. {\bf 140}
(2012), no.~12, 4063--4074.

\smallskip

\bibitem{tate}
J. Tate, Genus change in inseparable extensions of function
fields, Proc. Amer. Math. Soc. {\bf 3} (1952), 400--406.

\smallskip

\bibitem{totaro-1}
B. Totaro, Birational geometry of quadrics in characteristic 2, J.
Algebraic Geom. {\bf 17} (2008), no.~3, 577--597.

\smallskip

\bibitem{totaro-jump}
B. Totaro, Jumping of the nef cone for Fano varieties, J.
Algebraic Geom. {\bf 21} (2012), no.~2, 375--396.

\smallskip

\bibitem{wenzel}
C. Wenzel, Rationality of $G/P$ for a nonreduced parabolic
subgroup-scheme $P$, Proc. Amer. Math. Soc. {\bf 117} (1993),
no.~4, 899--904.

\smallskip

\bibitem{wies}
J. A. Wi\'sniewski, Rigidity of the Mori cone for Fano manifolds,
Bull. Lond. Math. Soc. {\bf 41} (2009), no.~5, 779--781.

\smallskip

\bibitem{xu}
C. Xu, On the base-point-free theorem of 3-folds in positive
characteristic, J. Inst. Math. Jussieu {\bf 14} (2015), no.~3,
577--588.

\end{thebibliography}
\end{document}